\documentclass[preprint,12pt]{elsarticle}
\usepackage{amsmath,amssymb,amsthm,color,url}
\newtheorem{thr}{Theorem}

\newtheorem{cor}[thr]{Corollary}

\newtheorem{con}[thr]{Conjecture}

\theoremstyle{definition}

\newtheorem{quest}[thr]{Question}

\theoremstyle{remark}

\begin{document}

\begin{frontmatter}

\title{Three observations on spectra of zero-nonzero patterns}

\author{Yaroslav Shitov}

\ead{yaroslav-shitov@yandex.ru}

\address{National Research University Higher School of Economics, 20 Myasnitskaya Ulitsa, Moscow 101000, Russia}

\begin{abstract}
Using standard techniques from combinatorics, model theory, and algebraic geometry, we prove generalized versions of several basic results in the theory of spectrally arbitrary matrix patterns. Also, we point out a counterexample to a conjecture proposed recently by McDonald and Melvin.
\end{abstract}


\end{frontmatter}

The study of spectra of zero-nonzero matrix patterns began more than a decade ago (see~\cite{BMOD, DJOD}), and a considerable amount of publications related to matrix completion problems are devoted to this topic. To begin with, recall that an $n\times n$ \textit{zero-nonzero pattern} is a matrix with entries $*$ and $0$. In other words, we can think of a zero-nonzero pattern as a class of $n\times n$ matrices which have non-zero elements at the same positions, which are indicated by the $*$ sign. Such a pattern $S$ is called \textit{spectrally arbitrary} with respect to a field $\mathbb{F}$ if any monic polynomial $f\in\mathbb{F}[t]$ of degree $n$ arises as the characteristic polynomial of a matrix with entries in $\mathbb{F}$ and pattern $S$.

Several well known results on this topic are usually being formulated in the case of the real numbers, and questions often arise as to whether or not the corresponding results are true over other fields. Many results of this kind admit natural generalizations for matrices over arbitrary fields, but some of these generalizations seem to remain unknown for the community. Questions of this kind occasionally appear as \textit{'open problems'} in the literature, so we believe it would be helpful to clarify the situation when possible.

\section{Employing elementary equivalence}

A \textit{first-order formula} in the language of fields is an expression constructed from polynomial equations with integer coefficients via the logical connectors and quantifiers. Such a formula is called a \textit{first-order sentence} if it does not contain free variables, that is, any sentence has a well-defined truth value over any field (although this value can, of course, depend on the field).

We will show that, for any $n\times n$ zero-nonzero pattern $S$, there is a first-order sentence $\Phi_S$ which is true over a field $\mathbb{F}$ if and only if $S$ is spectrally arbitrary with respect to $\mathbb{F}$. We denote by $\Sigma_S$ the set of all positions $(i,j)$ satisfying $S_{ij}=*$, and we introduce the two variables $x_{ij}, y_{ij}$ for any such $(i,j)$. We denote by $X_S$ the matrix obtained from $S$ by replacing every $*$ with a corresponding variable $x_{ij}$, and we denote by $\chi_{k,S}$ the $k$th coefficient of the characteristic polynomial of $X_S$. Clearly, the formula
$$\forall t_1\ldots\forall t_n\,\,\,
\exists (x_{ij})\exists (y_{ij})\,\,\,
\left(\bigwedge_{(i,j)\in\Sigma_S} x_{ij}y_{ij}=1\right)\wedge \left(\bigwedge_{k=1}^n\chi_{k,S}=t_k\right)
$$
is a desired first-order sentence $\Phi_S$. (Note that the equations $x_{ij}y_{ij}=1$ are needed to ensure that the $x_{ij}$'s are non-zero.)

Now we employ the standard result of model theory stating that a pair of algebraically closed fields of the same characteristic are \textit{elementarily equivalent}, that is, any first-order sentence takes the same truth value with respect to any of these fields. Also, the same result holds for any pair of fields that are \textit{real closed} (that is, fields over which $x^2+1$ is an irreducible polynomial and which become algebraically closed after adjoining the roots of this polynomial). The proofs and more thorough discussions of these results can be found in model theory textbooks such as~\cite{Eng}. We get the following corollaries. 

\begin{cor}\label{obst1}
A zero-nonzero pattern is spectrally arbitrary over $\mathbb{C}$ if and only if it is spectrally arbitrary over the algebraic closure of $\mathbb{Q}$.
\end{cor}

\begin{cor}\label{obst2}
A zero-nonzero pattern is spectrally arbitrary over $\mathbb{R}$ if and only if it is spectrally arbitrary over the real closure of $\mathbb{Q}$.
\end{cor}

Our discussion and, in particular, Corollary~\ref{obst1} answer Question~2 of Section~6 in~\cite{Mel} and improve the results of Corollaries~9 and~12 in~\cite{McM}. Corollary~\ref{obst2} confirms Conjecture~1.15 in~\cite{Mel}.

\section{Counting transcendence degrees}

An $n\times n$ zero-nonzero pattern $S$ is called \textit{irreducible} if, for any partition of the indexing set $\{1,\ldots,n\}$ into two disjoint subsets $(I,J)$, there are $i\in I$ and $j\in J$ such that $S_{ij}=*$. It is well known (see Lemma~2.3 in~\cite{BMOD}) that any such pattern $S$ has a set $U$ of $n-1$ non-zero positions such that every matrix with sign pattern $S$ is similar to a matrix with the same pattern and over the same field which has ones at the positions in $U$. A straightforward counting argument shows that no irreducible $n\times n$ zero-nonzero pattern with less than $2n$ non-zero elements can be spectrally arbitrary with respect to any finite field. (Due to the existence of irreducible polynomials of any degree over finite fields, the word 'irreducible' can be removed from the previous sentence. This result is attributed to Shader in~\cite{McM}, and it gives the finite field version of the so-called $2n$\textit{-conjecture}.)

Is the $2n$-conjecture true over any field? No. The paper~\cite{mysp} presents an explicit example of a $708\times 708$ pattern which contains only $1415$ non-zero entries and is spectrally arbitrary with respect to $\mathbb{C}$. However, the number of non-zero entries in \textit{irreducible} spectrally arbitrary patterns is bounded from below by $2n-1$. This result was stated and proved in~\cite{BMOD} over $\mathbb{R}$, but it is valid over any field. More than that, even the proof given there can be adapted for the general case.

\begin{thr}\label{thrt1}
Let $S$ be an $n\times n$ irreducible zero-nonzero pattern that is spectrally arbitrary over a field $\mathbb{F}$. Then $S$ has at least $2n-1$ non-zero entries.
\end{thr}

\begin{proof}
The validity of the result when $\mathbb{F}$ is finite was pointed out above, and we assume in the rest of the proof that $\mathbb{F}$ is infinite. Let $U$ be a set of positions as in the first paragraph of this section. We denote by $X$ the matrix obtained from $S$ by replacing the entries in $U$ by ones and replacing every of the other $k$ stars by the variables $x_1,\ldots,x_k$. By the definition of $U$, and since $S$ is spectrally arbitrary, every monic polynomial of degree $n$ over $\mathbb{F}$ can be realized as the characteristic polynomial of a matrix obtained from $X$ by replacing the variables by elements of $\mathbb{F}$.

Let $t^n+c_{n-1}t^{n-1}+\ldots+c_0$ be the characteristic polynomial of $X$, where the $c_i$'s belong to $\mathbb{F}[x_1,\ldots,x_{k}]$. If the result was false, we would have $k<n$, which would mean that the polynomials $c_0,\ldots,c_{n-1}$ are algebraically dependent over $\mathbb{F}$. In other words, there would exist a non-zero polynomial $\psi\in\mathbb{F}[\tau_1,\ldots,\tau_n]$ such that $\psi(c_0,\ldots,c_{n-1})=0$. As said in the above paragraph, the tuple $(c_0,\ldots,c_{n-1})$ can take arbitrary values in $\mathbb{F}^n$, so the polynomial $\psi$ vanishes over the whole $\mathbb{F}^n$. This is impossible because $\mathbb{F}$ is an infinite set, so we have reached a contradiction.
\end{proof}

This argument is almost explicit in the proof of Theorem~6.2 in~\cite{BMOD}. However, several recent works (see Theorem~15 in~\cite{McM} and Theorem~5.16, Corollary~5.17 in~\cite{Mel}) prove different special cases of Theorem~\ref{thrt1} with different techniques. Since none of these works mentions the possibility to generalize the result to the case of arbitrary fields, I believe it would be helpful to have an explicit proof of such a generalization in print.

\section{Taking an extension of a field}

Conjecture~6 in the recent paper~\cite{McM} stated the following.

\begin{con}\label{conMcM}
If $\mathbb{F}\subset\mathbb{K}$ is an extension of fields and a zero-nonzero pattern $S$ is spectrally arbitrary over $\mathbb{F}$, then $S$ is spectrally arbitrary over $\mathbb{K}$.
\end{con}

This conjecture is false. Since the $2\times 2$ pattern with all $*$'s is spectrally arbitrary over $\mathbb{C}$, and since every degree-four polynomial over $\mathbb{C}$ can be written as a product of two degree-two polynomials, the pattern
$$
\mathcal{D}=\left(\begin{array}{cccc}
*&*&0&0\\
*&*&0&0\\
0&0&*&*\\
0&0&*&*
\end{array}\right)
$$
is spectrally arbitrary with respect to $\mathbb{C}$. It is clear from the construction of $\mathcal{D}$ that the characteristic polynomial of any matrix with such a pattern is reducible over the field generated by the entries of the matrix. Therefore, there are no matrix $M$ with entries in the rational function field $\mathbb{C}(\xi)$ such that the pattern of $M$ is $\mathcal{D}$ and the characteristic polynomial of $M$ is $t^4+\xi$. By the way, this counterexample disproves a slightly more specific Conjecture~5.1 in~\cite{Mel} as well. However, we do not know if Conjecture~\ref{conMcM} remains false if we impose an additional assumption stating that $S$ is irreducible.

\section{Further work}

One of the problems that remain open was pointed out in the last sentence of the previous section. A more specific version of this question, asked by Melvin, looks quite intriguing.

\begin{quest}(See~\cite{Mel}.)
Does there exist a zero-nonzero pattern which is spectrally arbitrary with respect to $\mathbb{Q}$ but not with respect to $\mathbb{R}$?
\end{quest}

\begin{quest}(See~\cite{Mel}.)
Does there exist a zero-nonzero pattern which is spectrally arbitrary with respect to $\mathbb{R}$ but not with respect to $\mathbb{C}$?
\end{quest}

We recall that the $2n$-conjecture fails over $\mathbb{C}$ but remains open over $\mathbb{R}$. Similar questions over different fields are also of interest, and we believe that the following conjecture is within reach.

\begin{con}\label{concon}
Any $n\times n$ zero-nonzero pattern spectrally arbitrary with respect to $\mathbb{Q}$ has at least $2n$ nonzero elements.
\end{con}

As pointed out in~\cite{mysp}, the $2n$-conjecture over $\mathbb{C}$ fails mostly for the reason that the set of irreducible polynomials is not rich enough over that field. Since the rational irreducible polynomials can have arbitrary degree, one can assume without loss of generality that the pattern as in Conjecture~\ref{concon} is irreducible. Therefore, the technique used in~\cite{mysp} cannot invalidate Conjecture~\ref{concon}.

\end{document}